\RequirePackage{fix-cm}
\documentclass[smallextended]{svjour3}       

\smartqed  

\usepackage{graphicx}
\usepackage{amsmath, amscd, amsfonts, amssymb, graphicx, color}
\begin{document}

\title{Estimates of Norms on Krein Spaces\thanks{The first author is grateful to the National Institute of Technology Karnataka for the financial support. The present work of the second author was partially
		supported by National Board for Higher Mathematics (NBHM), Ministry of Atomic
		Energy, Government of India (Reference No.2/48(16)/2012/NBHM (R.P.)/R\&D 11/9133).}
}


\author{Athira Satheesh K. \and	P. Sam Johnson \and K. Kamaraj 
}


\institute{Athira Satheesh K. \at
	Department of Mathematical and Computational Sciences, 
	National Institute of Technology Karnataka, Surathkal, Mangaluru 575 025, India\\
	\email{athirachandri@gmail.com}           
	\and
	P. Sam Johnson \at
	Department of Mathematical and Computational Sciences, 
	National Institute of Technology Karnataka, Surathkal, Mangaluru 575 025, India\\
	\email{sam@nitk.edu.in} 
	\and 
	K. Kamaraj \at
	Department of Mathematics, University  College of Engineering Arni, Anna University, Arni 632 326, India.\\
	\email{krajkj@yahoo.com}
}
\date{Received: date / Accepted: date}

\maketitle

\begin{abstract}
	Various norms can be defined on a Krein space by choosing different underlying fundamental decompositions. Some estimates of norms on Krein spaces are  discussed and few results in Bognar's paper are generalized.
\end{abstract}

\section{Introduction}

Let $\mathcal K$ be a complex vector space with a Hermitian sesquilinear form defined on it. Then we call $(\mathcal K,[.,.])$ an inner product space. An element $x\in \mathcal K$ is called neutral, positive, or negative if $[x,x]=0$, $[x,x]>0$, or $[x,x]<0$ respectively. If $\mathcal K$ contains positive as well as negative elements, then it is called an indefinite inner product space, otherwise it is called a semi-definite inner product space. We refer \cite{BOGNAR,AZIZOV} for basics on indefinite inner product  spaces.  The concept of indefinite inner product was first found in a paper on quantum field theory by Dirac in 1942 \cite{dirac}. Pontrjagin gave the mathematical interpretation of indefinite inner product. Bognar \cite{BOGNAR1}, Hansen \cite{HANSEN}, Langer \cite{LANGER} et al. have investigated the notion of norm in indefinite inner product spaces.

An indefinite inner product space $(\mathcal K,[.,.])$ is decomposable if it can be written as an orthogonal direct sum of a  neutral subspace $\mathcal  K^0$, a positive definite subspace $\mathcal K^+ $ and a negative definite subspace $\mathcal  K^-$:
\begin{eqnarray}\label{e1} \mathcal  K=\mathcal K^0[\dot{+}]\mathcal K^+ [\dot{+}]\mathcal K^-. \end{eqnarray} Then (\ref{e1}) is known as a fundamental decomposition of $\mathcal K$. 

An indefinite inner product space $(\mathcal K,[.,.])$ is a Krein space if it can be written as an orthogonal direct sum of a positive definite subspace $\mathcal K^+$ and a negative definite subspace $\mathcal K^-$ such that $(\mathcal K^+,[.,.])$ and $(\mathcal K^-,-[.,.])$ are Hilbert spaces. Let a fundamental decomposition of a Krein space $\mathcal K$ be given by
\begin{eqnarray}\label{e2}\mathcal K=\mathcal K^+[\dot{+}]\mathcal K^- \end{eqnarray}
and $P^\pm$ be the orthogonal  projections onto $\mathcal K^\pm$. The linear map 
\begin{eqnarray*} 
	J=P^+-P^-,
\end{eqnarray*} 
is called the fundamental symmetry corresponding to $(\ref{e2})$. Then \begin{gather*} (x,y)_J=[Jx,y] \end{gather*} 
is a positive definite inner  product on $\mathcal K$, called $J$-inner product corresponding to the fundamental decomposition (\ref{e2}). We can write
\begin{eqnarray}\label{eq1}
(x,x)_{J}=[Jx,x]=[(2P^+-I)x,x]=2[P^+x,P^+x]-[x,x].
\end{eqnarray}
The corresponding norm (called $J$-norm) is denoted by 
\begin{gather*} \|x\|_J=(x,x)_J^\frac{1}{2}=[Jx,x]^\frac{1}{2}. \end{gather*}

Different fundamental decompositions induce different $J$-norms. Hence various norms can be defined on a Krein space by choosing different underlying fundamental decompositions.

A different fundamental decomposition of $\mathcal K$ say, $\mathcal K=\mathcal K_1^{+'}[\dot{+}]\mathcal K_2^{-'}$
makes the norm of an element $x$ larger than $|[x,x]|$. Roughly speaking, if the spaces $\mathcal K_1^{+'}$  and $\mathcal K_2^{-'}$ ``come closer'' to a neutral set $\mathcal K^0$, these norms in general be unbounded. It is interesting to observe that how the norm of a single element actually depends upon the choice of fundamental decomposition \cite{LANGER}. We end the section with some examples.  In the second section, some preliminary results are given which will be used in the sequel. The third section contains our main results concerning estimates of norms on Krein spaces.

\begin{example}\cite{DRITSCHEL}
	Minkowski space $M^{n+1}$ is defined as the set of $(n+1)$-dimensional column vectors $x=(x_1,x_2,\ldots,x_n,t_1)^t$ ($t$ indicates the transpose of a matrix)
	with complex entries and the indefinite inner product is defined by $$[x,y]=x_1\overline{y_1}+x_2\overline{y_2}+\dots+x_n\overline{y_n}+t_1 \overline{t_2}$$  where $x=(x_1,x_2,\ldots,x_n,t_1)^t$,  $ y=(y_1,y_2,\ldots,y_n,t_2)^t$ $\in M^{n+1}$. Then $M^{n+1}$  is a Krein space. A fundamental symmetry for the space is given by the matrix
	$$
	\begin{pmatrix}
	I_n&0\\
	0&-1
	\end{pmatrix},$$
	where $I_n$ denotes the identity matrix of order $n.$
\end{example}

\begin{example}
	Consider $\mathcal K=\ell_2$, the linear space of square-summable sequences, with $$[x,y]=\sum^\infty_{i=1} (-1)^ix_i \overline{y_i} \quad\text{for } \ x=(x_i)^\infty_{i=1}, \ y=(y_i)^\infty_{i=1}\in \mathcal K.$$ Let $\mathcal K^+ =\big\{(x_i)^\infty_{i=1}:x_i=0   \text{ if }  i  \text{ is odd}\big\}$  and $\mathcal K^- =\big\{(x_i)^\infty_{i=1}:x_i=0  \text{ if }   i \text{ is even}\big\}.$ 		
	Then $\mathcal K=\mathcal K^+[\dot{+}]\mathcal K^-$, where $\mathcal K^+$ and  $\mathcal K^-$ are complete with respect to the induced norm and hence $\mathcal K$ is a Krein space.
\end{example}
\begin{example} 
	Consider $\mathcal K=C[-1,1]$ the linear space of all complex-valued continuous functions
	defined on the interval $[-1,1]$  with $$[x,y]=\int_{-1}^{1}x(t)\overline{y(-t)}dt \quad \text{for } \  x,y\in \mathcal K.$$ Then $\mathcal K$ admits a fundamental decomposition $\mathcal K=\mathcal K^+[\dot{+}]\mathcal K^-$ where $\mathcal K^+$ and $\mathcal K^-$ the spaces of all continuous even and odd functions on $[-1,1]$ respectively,  are complete with respect to the induced norm and hence $\mathcal K$ is a Krein space.
\end{example}

\begin{example}
	\cite{Kaltenback} Let $\Omega$ be a set and $\Sigma$ be a $\sigma$-algebra on $\Omega$. Let $\mu_+$ and $\mu_-$ be two mutually singular positive measures defined on $\Sigma$. Set $\mu=\mu_++\mu_-$. Define
	\begin{gather*}     [f,g]=\int_{\Omega}f\overline{g}d\mu_+-\int_{\Omega}f\overline{g}d\mu_- \quad \text{for } \ f,g\in L^2(\mu).\end{gather*} Then $L^2(\mu)=L^2(\mu_+)[\dot{+}]L^2(\mu_-)$ forms a fundamental decomposition, since $(L^2(\mu_+),[.,.])$ and $(L^2(\mu_-),-[.,.])$ are Hilbert spaces. Thus $L^2(\mu)$ is a Krein space.
\end{example}

\section{Preliminaries}

\begin{theorem} \cite{AZIZOV}
	Let $\mathcal K$ be a Krein space. Then $\mathcal K$ has several fundamental decompositions with non-zero components.  All norms induced by different fundamental decompositions are equivalent and hence they induce the same topology.
\end{theorem}

\begin{theorem}\cite{BOGNAR}
	Let $(\mathcal K,[.,.])$ be a Krein space. Then the following are equivalent:
	\begin{enumerate}
		\item[(1)] There exists a fundamental decomposition of $\mathcal K$.
		\item[(2)] There exists a maximal uniformly positive ortho-complemented subspace.
		\item[(3)] There exists a maximal uniformly negative ortho-complemented subspace.
	\end{enumerate}
\end{theorem}





\begin{example}\label{eg1}
	Let $\mathcal K$ be a two-dimensional vector space with basis $\{e_1,e_2\}$ and an indefinite inner product defined by $[e_1,e_1]=1,[e_2,e_2]=-1$ and $[e_1,e_2]=0$. If we take $Y=\text{span}\{e_1\}$, then it is a maximal uniformly positive definite subspace and hence there exists a fundamental decomposition of $\mathcal K$ with $\mathcal K^+=Y$ and $\mathcal K^-=\text{span}\{e_2\}$. Choosing $\mathcal K_n^+=\text{span}\{(n,1)\}$ and $ \mathcal K_n^-=\text{span}\{(1,n)\}$ where $n>1$, we get several fundamental decompositions. The corresponding fundamental symmetries  $J_n$ are given by $$
	J_n= \begin{pmatrix}
	\frac{n^2+1}{n^2-1}&\frac{-2n}{n^2-1}\\
	\frac{2n}{n^2-1}& \frac{-(n^2+1)}{n^2-1}
	\end{pmatrix}.
	$$
	Here we can see that the fundamental symmetries  $J_n$ satisfy $J_n^2=I$, $[J_nx,y]=[x,J_ny]$ and $[J_nx,J_ny]=[x,y]$ for all $x,y \in \mathcal K$.
\end{example}

\section{Main Results}
\begin{theorem}\cite{LANGER}\label{th0}
	Assume that $\mathcal K$ is a Krein space such that $[.,.]$ is indefinite and let $x\in \mathcal K$, $x\ne 0$. Then the following holds.
	\begin{enumerate}
		\item[(i)]  If $[x, x]\ne 0$, then 
		\begin{equation}
		\Big\{\|x\|_J:\text{J is a fundamental symmetry}\Big\}=[|[x,x]|^\frac{1}{2},\infty).
		\end{equation}
		Moreover,
		\begin{equation*} \|x\|_J=|[x,x]|^\frac{1}{2}
		\quad\text{if and only if}\quad x\in \mathcal  K^+_{J} \cup \mathcal K^-_{J},
		\end{equation*}
		where $\mathcal K=\mathcal K^+_{J}[\dot{+}]\mathcal K^-_{J}$ is the fundamental decomposition associated with $J$.
		\item[(ii)]  If $[x,x]=0$, then
		\begin{equation*}
		\Big\{\|x\|_J:\text{J is a fundamental symmetry}\Big\}=(0,\infty).
		\end{equation*}
	\end{enumerate}
\end{theorem}
\begin{theorem}\label{th1}
	
	Assume that $(\mathcal K,[.,.])$ is a Krein space and let $0\neq x\in \mathcal K.$
	\begin{enumerate}
		\item[(a)] If $[x,x]\neq 0$,  then for each real number $a>|[x,x]|^\frac{1}{2}$   there exists a fundamental symmetry $J_a$  such that $\|x\|_{J_a}=a$.
		\item[(b)] If $[x,x]=0$,  then for each positive real number a there exists a fundamental symmetry $J_a$  such that $\|x\|_{J_a}=a$.
\end{enumerate}\end{theorem}

\begin{proof}
	(a) Let  $[x,x] <0$ .
	Let $\mathcal K=\mathcal M^+[\dot{+}]\mathcal M^-$  be a fundamental decomposition such that $x\in \mathcal M^-$. Choose $0\neq y\in \mathcal M^+$ and let $\mathcal L^+$ and $\mathcal  L^-$ be subspaces such that
	\begin{align*}\mathcal  M^+&=\mathcal  L^+[\dot{+}]\text{span}\{y\} &   \mathcal  M^-= \mathcal L^-[\dot{+}]\text{span}\{x\}.
	\end{align*}
	Consider 
	$u(s)=sy+(1-s)x, $  $s\in [0,1].$
	We have $[u(0),u(0)]<0, [u(1),u(1)]>0$ and $[.,.]$ is continuous. Hence there exists $s_0\in(0,1)$ such that $[u(s_0),u(s_0)]=0.$ Let $z=u(s_0)$, then 
	\begin{align*}
	[z,z]&=0, &  [y,z]&=s_0[y,y]>0,& [z,x]&=(1-s_0)[x,x]<0.
	\end{align*}
	Let $v(t)=ty+(1-t)z$, $t\in (0,1],$ which is a positive element for $t\in (0,1].$ Now set $$\mathcal K^+_{t} =\mathcal  L^+ [\dot{+}] \text{span}\{v(t)\}$$ since $y $ and $z$ are orthogonal to $\mathcal  L^+$. Thus the orthogonal projection $P^+_{t}$ onto $\mathcal K^+_{t}$ can be written as 
	$P^+_{t}=P_{\mathcal L_+}+P_{v(t)}$, where $P_{\mathcal L_+}$ is the orthogonal projection onto $\mathcal L_+$ and $P_{v(t)}$ is the orthogonal projection onto $\text{span}\{v(t)\}$. We also have \begin{eqnarray}P_{v(t)}u=\frac{[u,v(t)]}{[v(t),v(t)]}v(t)
	\end{eqnarray}for any $u \in \mathcal K$.
	Let $w(t)$ be a non-zero element in $\text{span}\{y,x\}$ which is orthogonal to $v(t)$ and hence negative. With $\mathcal K^-_{t} = \mathcal L^+ [\dot{+}] \text{span}\{w(t)\}$, we have a fundamental decompositon $\mathcal K=\mathcal K^+_{t}[\dot{+}]\mathcal K^-_{t}$ and a corresponding fundamental symmetry $J_t=2P^+_{t}-I$. Now we get 
	\begin{eqnarray}
	[P^+_{t}x,P^+_{t}x]= [P^+_{t}x,x]=\frac{|[x,v(t)]|^2}{[v(t),v(t)]}=\frac{(1-t)^2|[x,z]|^2}{t^2[y,y]+2t(1-t)[y,z]}.
	\end{eqnarray}
	
\noindent	From equation (\ref{eq1}) we get,
	
	\begin{eqnarray}\label{eq2}
	(x,x)_{J_t}=\frac{2(1-t)^2|[x,z]|^2}{t^2[y,y]+2t(1-t)[y,z]}-[x,x].
	\end{eqnarray}
	The construction of $(x,x)_{J_t}$ in equation (\ref{eq2}) is taken from the proof of the Theorem (\ref{th0}). The details are given for the sake of completeness of the proof. 
	
	As $t$ varies in (0,1], $(x,x)_{J_t} $ takes all values in 
	$[|[x,x]|,\infty)$. Thus $\|x\|_{J_t}$ takes all values in  $[|[x,x]|^\frac{1}{2},\infty)$. Let $a\in[|[x,x]|^\frac{1}{2},\infty$) be such that $a^2=b>|[x,x]|$. Now let us try to find $t\in(0,1]$
	for which $(x,x)_{J}=b$ so that  $\|x\|_{J}=a$.
	
\noindent 	From (\ref{eq2}) and $(x,x)_{J}=b$ we get, 
	\begin{eqnarray}\label{eq3}
	\frac{2(1-t)^2|[x,z]|^2}{t^2[y,y]+2t(1-t)[y,z]}-[x,x]=b.
	\end{eqnarray}
	We have  $[x,x]<0$, $[y,y]>0$ and $[y,z]>0$. So let $h=[y,y]$. Replacing $y$ by $\frac{y}{\sqrt{h}}$ we get $[y,y]=1.$ Now set $A=|[x,z]|^2$, $B=[y,y]=1$, $C=[y,z]$, $D=[x,x].$ Thus from (\ref{eq2}) we get  $\frac{2(1-t)^2A}{t^2+2t(1-t)C}-D=b$ which implies
	$2A(1-2t+t^2)=(b+D) (t^2-2Ct^2+2Ct)$
	so that 
	\begin{eqnarray*} 
		t^2[(b+D)(1-2C)-2A]+t[(b+D)2C+4A]-2A=0,
	\end{eqnarray*}  
	which is a quadratic equation in $t$ whose discriminant is $4C^2(b+D)^2+8A(b+D)$, which is positive as $ b+D$ and $A$ are positive. Thus there exists $t\in (0,1]$ such that it satisfies equation (\ref{eq3}). Let us denote it by $t_b$.
	Then the subspaces
	\begin{eqnarray*} 
		\mathcal  K^+_{t}=  \mathcal L^+[\dot{+}]\text{span}\{v(t_b)\}, \  \mathcal K^-_{t}= \mathcal L^-[\dot{+}]\text{span}\{w(t_b)\} 
	\end{eqnarray*} 
	give a fundamental symmetry  corresponding to $t_b$. We denote it by $J_a$ and hence we see that  $(x,x)_{J_a}=b$ and $\|x\|_{J_a}=a$.\\
	Let $[x,x]>0 $. Then choose a fundamental decomposition $ \mathcal K= \mathcal M^+[\dot{+}] \mathcal M^-$ such that $x\in  \mathcal M^+$ and continue the proof as discussed above.
	
	(b) Let $[x,x]=0$. 
	Let $y$ be another neutral element that satisfies $[x,y]=1$. Define $u=\frac{1}{\sqrt{2}}(x+y)$, $v=\frac{1}{\sqrt{2}}(x-y)$. Then $x=\frac{1}{\sqrt{2}}(u+v)$, $[u,u]=1$, $[v,v]=-1$ and $[u,v]=0$. 
	Let $  \mathcal K= \mathcal M^+[\dot{+}] \mathcal M^-$ be a fundamental decomposition such that $ \mathcal M^+= \mathcal L^+[\dot{+}]\text{span}\{u\}$, $ \mathcal M^-= \mathcal L^-[\dot{+}]\text{span}\{v\}$ with some subspaces $ \mathcal L^\pm$. Set 
	\begin{eqnarray*} 
		w(t)=u+tv, t\in(-1,1).
	\end{eqnarray*} 
	Then $[w(t),w(t)]=(1-t^2)>0$, $t\in(-1,1)$. Hence $ \mathcal K_{t,+}= \mathcal L^+[\dot{+}]\text{span}\{w(t) \}$ is a maximal uniformly positive subspace. Now the projection $P_{t,+}$ onto $ \mathcal K_{t,+}$ can be written as $P_{t,+}=P_{ \mathcal L_+}+P_{w(t)}$ and we get 
	\begin{gather*}
	[P_{t,+}x,P_{t,+}x]= [P_{t,+}x,x]=\frac{|[x,w(t)]|^2}{[w(t),w(t)]}=\frac{1-t}{2(1+t)},
	\end{gather*}
	which takes all the values in $(0,\infty)$ if $t$ varies in $(-1,1)$. Let $a\in (0,\infty$) be such that $a^2=b.$ Thus from (\ref{eq1}) and solving $(x,x)_{J_t}=b$ we get  $2[P_{t,+}x,P_{t,+}x]=b$.  That is, $\frac{2(1-t)}{2(1+t)}=b$ and hence $t=\frac{1-b}{1+b}$. 
	Thus for $t_b=\frac{1-b}{1+b},$ $$\mathcal K^+_{t_b}=\mathcal L^+[\dot{+}]\text{span}\{w(t_b) \}$$ is a maximal uniformly positive subspace and hence there exists a fundamental decomposition. We denote the corresponding fundamental symmetry by $J_{a}$. Hence we get that $(x,x)_{J_a}=b$ and $\|x\|_{J_a}=a$.
\end{proof}

\begin{corollary}
	Let  $0\neq x\in  \mathcal K$, $[x,x]\neq 0$ and  $J$ be a given fundamental symmetry. Then there exists a fundamental symmetry $K$ such that 
	\begin{gather*}
	\|x\|_{J}<\|x\|_{K}.
	\end{gather*}
\end{corollary}
\begin{proof}
	Choose a positive real number  $k>\|x\|_{J}$. Then by Theorem {\ref{th1}}(a) there exists a fundamental symmetry $K$ such that  $ \|x\|_{K}=k$, which is greater than $ \|x\|_{J}$ and that implies  $\|x\|_{J}<\|x\|_{K}$.
\end{proof}

\begin{corollary}
	Let  $0\neq x\in  \mathcal K$ be a neutral element and  $J$  be a given  fundamental symmetry. Then there exist fundamental symmetries $K_1$ and  $K_2$ such that 
	\begin{gather*}
	\|x\|_{K_1}<\|x\|_{J}<\|x\|_{K_2}.
	\end{gather*}
\end{corollary}

\begin{proof}
	Choose  positive real numbers $k_1,k_2$ such that $0<k_1<\|x\|_{J}<k_2$. Then by Theorem {\ref{th1}} (b) there exist fundamental symmetries $K_1$ and $K_2$ such that $\|x\|_{K_1}<\|x\|_{J}<\|x\|_{K_2}$.
\end{proof}

\begin{corollary}
	Let $x$ be any arbitrary non-zero (neutral or non-neutral) element in $\mathcal K$. If   $\|x\|_{J_1}<\|x\|_{J_2}$ for some  fundamental symmetries $J_{1}$ and $J_{2}$, then there exists a fundamental symmetry $J$ such that $\|x\|_{J_1}<\|x\|_{J}<\|x\|_{J_2}$.
\end{corollary}

\begin{proof}
	Choose $a>|[x,x]|^\frac{1}{2}$ such that $\|x\|_{J_1}<a<\|x\|_{J_2}$. Then by Theorem {\ref{th1}} the result follows.
\end{proof}

\begin{corollary}\label{th2}
	Let $\mathcal K$ be a Krein space with a fundamental decomposition and a corresponding fundamental symmetry $J$. Let $(x_n)$ be a sequence of non-zero neutral or non-neutral vectors such that $(x_n)$ converges to some $x$ in $\mathcal K$. Then there exists a sequence of fundamental symmetries $(J_n)$ such that $\|x_n\|_{J_n}\rightarrow\infty$.\end{corollary}

\begin{proof} 
	Let $J_1=J$. Now  choose a real number $a_2> \max \{|[x_2,x_2]|^\frac{1}{2},\|x_1\|_{J_1}\}$. Then by Theorem {\ref{th1}} (a) there exists a fundamental symmetry $J_2$ such that $\|x_2\|_{J_2}=a_2>\|x_1\|_{J_1}$.
	In a similar way, we can find $J_k$  by  choosing 
	\begin{gather*} a_k>\max\big\{|[x_k,x_k]|^\frac{1}{2},\|x_{k-1}\|_{J_{k-1}}\big\}\end{gather*} and by using Theorem {\ref{th1}} (a)  we get $J_k$ such that  $\|x_k\|_{J_k}=a_k>\|x_{k-1}\|_{J_{k-1}}$. Continuing the process we see that  $\|x_n\|_{J_n}\rightarrow\infty$ as $n\rightarrow\infty$.
\end{proof}

\begin{corollary}\label{th3}
	Let $(x_n)$ be a sequence of non-zero neutral elements in $\mathcal K$. Then there exists a sequence of fundamental symmetries  $(J_n)$ such that  $\|x_n\|_{J_n}\rightarrow0$ as $n\rightarrow\infty$.
\end{corollary}

\begin{proof}
	We have $x_1\neq0$, then we can find a fundamental symmetry $J_1$ such that $\|x_1\|_{J_1}>0$. Choose a real number $a_2$ such that $\|x_1\|_{J_1}>a_2>0$. Then by Theorem {\ref{th1}} (b)  there exists a fundamental symmetry $J_2$ such that 
	$\|x_2\|_{J_2}=a_2$ and so we get $\|x_1\|_{J_1}>\|x_2\|_{J_2}$. Choose a real number $a_k$ such that 
	\begin{eqnarray*} 
		\|x_{k-1}\|_{J_{k-1}}>a_k>0.
	\end{eqnarray*} 
	Then by Theorem {\ref{th1}} (b)  there exists a fundamental symmetry $J_k$ such that
	$\|x_{k-1}\|_{J_{k-1}}>\|x_k\|_{J_k}$. Thus we see that $(\|x_n\|_{J_n})$ is a decreasing sequence which is bounded below by $0$ and hence $\|x_n\|_{J_n}\rightarrow0$ as $ n\rightarrow\infty$.
\end{proof}
\begin{remark}
	Corollaries {\ref{th2}} and {\ref{th3}} generalize the Lemma in \cite{BOGNAR1} which says for a non-neutral element $x$ there exists a sequence of fundamental norms $(p_n)$ such that $p_n(x)\rightarrow\infty$ as $n\rightarrow\infty$ and for a neutral element $y$ there exist sequences of fundamental norms $(p_n)$ and $(q_n)$ such that $p_n(x)\rightarrow\infty$ and $q_n(x)\rightarrow 0$ as $n\rightarrow\infty$.
\end{remark}

\begin{example}
	Consider the fundamental symmetry $J_n$ given in Example \ref{eg1}. Then for $x=(x_1,y_1)\in  \mathcal K$ $$\|x\|_{J_n}^2=\frac{(n^2+1)((x_1^2+y_1^2)-4nx_1y_1}{n^2-1}.$$ We fix $x=(2,1)$. Then $[x,x]=3$. Let $a=2>|[x,x]|^\frac{1}{2}$. By solving  $\|x\|_{J_n}^2=4$, we see that $n$ equals to the positive square root of the equation $5n^2-8n-7=0$. 
\end{example}

\begin{theorem} Let $\mathcal K$ be a Krein space. Then the following are true.
	\begin{enumerate}
		\item[(a)]Let $0\ne x\in \mathcal K$ be a non-neutral element, $\alpha \in \mathbb{C}$. Then for every $\varepsilon$ such that $\varepsilon>|[x,x]|^\frac{1}{2} |1-|\alpha||$, there exists  a fundamental symmetry $J$ such that $|\|x\|_J-\|\alpha x\|_J|<\varepsilon.$ 
		
		\item[(b)] 
		Let $0\ne x\in \mathcal K$ be a neutral element, $\alpha \in \mathbb{C}$. Then for every $\varepsilon>0$, there exists a fundamental symmetry $J$ such that $|\|x\|_J-\|\alpha x\|_J|<\varepsilon.$
	\end{enumerate}
	\begin{proof}
		Suppose $|\alpha|=1$, the result is trivial.
		Next we asssume that $|\alpha|\ne 1$. 
		\begin{enumerate}
			\item[(a)] We have $\varepsilon>|[x,x]|^\frac{1}{2} |1-|\alpha||$, which implies $ |[x,x]|^\frac{1}{2}<\frac{\varepsilon}{|1-|\alpha||}$. Let $c\in \mathbb{R}$ be such that $ |[x,x]|^\frac{1}{2}<c<\frac{\varepsilon}{|1-|\alpha||}$. Then by Theorem \ref{th1}  there exists  a fundamental symmetry $J$ such that $\|x\|_J=c$ which implies $\|x\|_J<\frac{\varepsilon}{|1-|\alpha||}$ so that we get $|\|x\|_J-\|\alpha x\|_J|<\varepsilon.$ 
			
			\item[(b)]We have $|1-|\alpha||>0$, which implies $\frac{\varepsilon}{|1-|\alpha||}>0$. Let $c\in \mathbb{R}$ be such that $0<c<\frac{\varepsilon}{|1-|\alpha||}$. Then by Theorem \ref{th1}  there exists  a fundamental symmetry $J$ such that $\|x\|_J=c$ which implies $\|x\|_J<\frac{\varepsilon}{|1-|\alpha||}$ so that we get $|\|x\|_J-~\|\alpha x\|_J|<\varepsilon.$ 
		\end{enumerate}
	\end{proof}
	
	
\end{theorem}
\begin{theorem}\label{th4}
	Let $x$ and $y$ be orthogonal non-neutral elements of a Krein space $\mathcal K$ with a fundamental decomposition $\mathcal K=\mathcal K^+[\dot{+}]\mathcal K^-$.  If $x$ and $y$ are linearly independent and if 	
	 \begin{eqnarray}\label{c1}dim(\mathcal K^+)>1,dim(\mathcal K^-)>0,[y,y]>0 \end{eqnarray} or 
	\begin{eqnarray}\label{c2} dim(\mathcal K^-)>1,dim(\mathcal K^+)>0,[y,y]<0 \end{eqnarray}then there exists a sequence of fundamental symmetries $(J_n)$ such that $\frac{\|y\|_{J_n}}{\|x\|_{J_n}}\rightarrow0$ as  $ n\rightarrow\infty$.
\end{theorem}
\begin{proof}
	The case (\ref{c2}) can be reduced to (\ref{c1}) by passing to the inner product $[u,v]'=-[u,v]$ where $u,v\in \mathcal K$. Thus we consider only the case (\ref{c1}).
	From the hypothesis,  we can find at least two positive elements $x_1,x_2$ and a negative element $y_1$ in $\mathcal K$ such that\\ 
	\begin{equation*} \mathcal K=\mathcal L_1^+ [\dot{+}] \text{span}\{x_1\} [\dot{+}] \text{span}\{x_2\} [\dot{+}] \mathcal L_2^- [\dot{+}] \text{span}\{y_1\},
	\end{equation*}
	where $\mathcal L_1^+$ and $\mathcal L_2^-$ are positive and negative subspaces respectively.  
	
We now first discuss the case when $[x,x]>0$. Choose $x_1=y$ and and $x_2=\frac{x}{\sqrt{[x,x]}}$ so that $[x_2,x_2]=1$ and choose $y_1$ such that $[y_1,y_1]=-1$. We can find a neutral element $e_1=s_0x_2+(1-s_0)y_1$ for some $s_0\in (0,1)$. \
	Take $ v(t_n)=t_nx_2+(1-t_n)e_1$ where $t_n=\frac{1}{n},n>1$. Then $[v(t_n),v(t_n)]={t_n}^2+2s_0t_n(1-t_n)>0$ and $[v(t_n),x_1]=0$. Set
	\begin{equation*} {\mathcal K_n}^+=\mathcal L_1 [\dot{+}] \text{span}\{x_1\} [\dot{+}] \text{span}\{v(t_n)\} .
	\end{equation*}
	Thus the orthogonal projection ${P_n}^+$ onto  ${\mathcal K_n}^+$ can be written as  
	\begin{equation*}
	{P_n}^+=P_{\mathcal L_1}+P_{x_1}+P_{v(t_n)}
	\end{equation*}
	where $P_{\mathcal L_1}$ is the orthogonal projection onto $\mathcal L_1$, $P_{x_1}$ is the orthogonal projection onto  $\text{span}\{x_1\}$ and $P_{v(t_n)}$ is the orthogonal projection onto  $\text{span}\{v(t_n)\}$, which has the form 
	\begin{equation*}
	P_{v(t_n)}z=\frac{[z,v(t_n)]}{[v(t_n),v(t_n)]}v(t_n).
	\end{equation*}
	Choosing a non-zero element $u(t_n)$ in the $\text{span}\{x_2,y_1\}$, which is orthogonal to $v(t_n)$, we get a fundamental decompositon with ${\mathcal K_n}^+$ and ${\mathcal K_n}^-= \mathcal L_2[\dot{+}] \text{span}\{u(t_n)\} $ and a corresponding fundamental symmetry $J_n=2{P_n}^+-I$. 
	For a vector $z\in \mathcal K$ we have
	\begin{equation*}
	{\|{z\|^2_{J_n}}}=[J_nz,z]=[(2{P_n}^+-I)z,z]=2[{P_n}^+z,z]-[z,z].
	\end{equation*}
	Let us calculate ${\|y\|^2_{J_n}}$. Since  $x_1=y$ we have
	$ {P_n}^+y=P_{x_1}y=y$.
	Thus $\|{y\|_{J_n}}={[y,y]}^\frac{1}{2}$ for all $n>2$. 
	Let us find $\|{x\|_{J_n}}$.We have
	\begin{equation*}
	{P_n}^+x_2=P_{v(t_n)}x_2=\frac{[x_2,v(t_n)]}{[v(t_n),v(t_n)]}v(t_n)=\frac{t_n+s_0(1-t_n)}{{t_n}^2+2s_0t_n(1-t_n)}v(t_n),
	\end{equation*}
	which implies
	\begin{equation*}
	\begin{split}
	[{P_n}^+x_2,x_2]&=\frac{t_n+s_0(1-t_n)}{{t_n}^2+2s_0t_n(1-t_n)}[v(t_n),x_2]\\&=1+\frac{{s_0}^2}{{t_n}^2+2s_0t_n(1-t_n)}+\frac{t_n-2}{{t_n}^2+2s_0t_n(1-t_n)}.
	\end{split}
	\end{equation*}
	And hence $\|{x\|_{J_n}}=|\sqrt{[x,x]}|\|{x_2\|_{J_n}}\rightarrow\infty$ as $n\rightarrow\infty.$ Thus $ \frac{\|{y\|_{J_n}}}{\|{x\|_{J_n}}}\rightarrow0$ as $n\rightarrow\infty.$
	
	We now discuss the case when $[x,x]<0$. We take $y_1=\frac{x}{\sqrt{|[x,x]|}}$ so that $[y_1,y_1]=-1$. Choose $x_2$ such that $[x_2,x_2]=1$. Proceeding as above we find $\|{y\|_{J_n}}={[y,y]}^\frac{1}{2}$ for all $n>2$ and $\|{y_1\|_{J_n}}$ as follows. We have
	\begin{equation*}
	{P_n}^+y_1=P_{v(t_n)}y_1=\frac{[y_1,v(t_n)]}{[v(t_n),v(t_n)]}v(t_n)=\frac{t_n+s_0(1-t_n)}{{t_n}^2+2s_0t_n(1-t_n)}v(t_n).
	\end{equation*}
	which implies
	\begin{equation*}
	\begin{split}
	[{P_n}^+y_1,y_1]&=\frac{s_0-1}{{t_n}^2+2s_0t_n(1-t_n)}[v(t_n),y_1]\\&=1+\frac{{s_0-1}^2}{{t_n}^2+2s_0t_n(1-t_n)}+\frac{t_n-2}{{t_n}^2+2s_0t_n(1-t_n)}.
	\end{split}
	\end{equation*}
	Thus $\|{x\|_{J_n}}=|\sqrt{|[x,x]|}|\|{y_1\|_{J_n}}\rightarrow\infty$ as $n\rightarrow\infty.$ And hence $ \frac{\|{y\|_{J_n}}}{\|{x\|_{J_n}}}\rightarrow0$ as $n\rightarrow\infty.$\\
\end{proof}


\begin{theorem}
	Let $x$ and $y$ be linearly independent elements of a Krein space $\mathcal K$ which are non-orthogonal. If $y$ is neutral, there exists a sequence of fundamental symmetries $ (J_n)$
	such that $\frac{\|y\|_{J_n}}{\|x\|_{J_n}}\rightarrow0$ as  $
	n\rightarrow\infty$.
\end{theorem}
\begin{proof}
	Since $[y,y]=0$, there exists a sequence of fundamental symmetries
	$(J_n)$ such that $ \|y\|_{J_n}\rightarrow0$ as  $
	n\rightarrow\infty$. 	We first discuss the case when $x$ is non-neutral.	By Theorem \ref{th0} we have 
	\begin{equation*}
		\Big\{\|x\|_J: J \text{ is a fundamental symmetry}\Big\}=[|[x,x]|^\frac{1}{2},\infty).
	\end{equation*}
	So for all n, $\|x\|_{J_n}\ge |[x,x]|^\frac{1}{2}$. We get
	$\Big(\frac{1}{\|x\|_{J_n}}\Big)$ is bounded and hence we can conclude that
	$\frac{\|y\|_{J_n}}{\|x\|_{J_n}}\rightarrow0$ as
	$n\rightarrow\infty$.
	
	We now discuss the case when $x$ is neutral. Let $[x,y]=k$. Since $x$ and $y$ are non-orthogonal, $k\neq 0$. By replacing $y$ by $\frac{y}{\overline{k}}$ we get $[x,y]=1.$ Let
	\begin{equation*}
	x_1=\frac{1}{\sqrt{2}}(x+y), y_1=\frac{1}{\sqrt{2}}(x-y), 
	\end{equation*} 
	then
	\begin{equation*}
	x=\frac{1}{\sqrt{2}}( x_1+y_1), y=\frac{1}{\sqrt{2}}( x_1-y_1), [ x_1, x_1]=1, [y_1,y_1]=-1, [ x_1,y_1]=0. 
	\end{equation*}
	Let $\mathcal K=\mathcal M^+[\dot{+}]\mathcal M^-$ be a fundamental decomposition such that
	\begin{equation*}
	\mathcal M^+=\mathcal L^+[\dot{+}]\text{span}\{ x_1\}, \mathcal M^-=\mathcal L^-[\dot{+}]\text{span}\{y_1\} 
	\end{equation*} 
	with some subspaces $\mathcal L^\pm$.  Set $v(t_n)=x_1+t_ny_1$, $t_n\in(-1,1)$. Then $$[v(t_n),v(t_n)]=1-{t_n}^2.$$
	We have $\mathcal K^+_{t_n}=\mathcal L^+[\dot{+}]\text{span}\{v(t_n) \}$ is a maximal uniformly positive subspace and hence there exists a fundamental decomposition of $\mathcal K$ with $\mathcal K^+=\mathcal K^+_{t_n}$. Now the projection $P^+_{t_n}$ onto $\mathcal K^+_{t_n}$ can be written as
	\begin{equation*}
	P^+_{t_n}=P_{\mathcal L_+}+P_{v(t_n)}. 
	\end{equation*}
	Thus
	\begin{equation*}
	[P^+_{t_n}x,x]=\frac{{|[x,v(t)]|}^2}{[v(t_n),v(t_n)]}=\frac{1-t_n}{2(1+t_n)},
	\end{equation*}
	from which we get 
	\begin{equation*}
	{\|{x\|_{J_n}}}^2=[J_nx,x]=2[{P_n}^+x,x]-[x,x]=\frac{2(1-t_n)}{2(1+t_n)}\rightarrow\infty
	\end{equation*}
	if we choose $(t_n)$ such that  $t_n\rightarrow-1$ as $n\rightarrow \infty$. Similarly we get 
	\begin{equation*}
	{\|{y\|_{J_n}}}^2=\frac{1+t_n}{1-t_n}\rightarrow 0
	\end{equation*}
	if we choose $(t_n)$ such that  $t_n\rightarrow-1$ as $n\rightarrow \infty$. Thus we see that $\frac{\|y\|_{J_n}}{\|x\|_{J_n}}\rightarrow0$ as	$n\rightarrow\infty$.\\
\end{proof}

\begin{example}
	Consider the two dimensional Minkowski space $\mathcal K=\mathbb R^2$ with the inner product $[x,y]=x_1y_1-x_2y_2$ where $x=(x_1,x_2)$, $y=(y_1,y_2)\in \mathbb R^2$. Consider the fundamental decompositions with $\mathcal K_n^+=\text{span}\{(\frac{n+1}{n},\frac{n-1}{n})\}$ and $\mathcal K_n^-=\text{span}\{(\frac{n-1}{n},\frac{n+1}{n})\}$ where $n>1$. Then we get $$\|x\|_{J_n}^2=\frac{1}{4}[(2n+2/n)(x_1^2+y_1^2)+4x_1y_1(1/n-n)].$$	
	Let $y=(1,1)$ and $x=(1,0)$. Then $\|y\|_{J_n}^2=\frac{2}{n}$ and $\|x\|_{J_n}^2=\frac{1}{2}(n+\frac{1}{n})$. Thus $\frac{\|y\|_{J_n}}{\|x\|_{J_n}}\rightarrow0$ as  $n\rightarrow\infty$.
	
\end{example}

\bibliographystyle{spmpsci}
\bibliography{hypo_EP}

\begin{thebibliography}{1}
\providecommand{\url}[1]{{#1}}
\providecommand{\urlprefix}{URL }
\expandafter\ifx\csname urlstyle\endcsname\relax
  \providecommand{\doi}[1]{DOI~\discretionary{}{}{}#1}\else
  \providecommand{\doi}{DOI~\discretionary{}{}{}\begingroup
  \urlstyle{rm}\Url}\fi

\bibitem{AZIZOV}
Azizov, T.Y., Iokhvidov, I.S.: Linear operators in spaces with indefinite
  metric and their applications.
\newblock A Wiley-Interscience Publication (1989)

\bibitem{BOGNAR}
Bognar, J.: Indefinite inner product spaces.
\newblock Grundlehren Der Mathematischen Wissenschaften in Einzeldarst.
  Springer-Verlag Berlin Heidelberg (1974).
\newblock \doi{10.1007/978-3-642-65567-8}.
\newblock \urlprefix\url{http://dx.doi.org/10.1007/978-3-642-65567-8}.
\newblock Reprint of the 1971 edition

\bibitem{BOGNAR1}
Bogn\'{a}r, J.: Various norms on indefinite inner product spaces.
\newblock Period. Math. Hungar. \textbf{6}(4), 309--321 (1975).
\newblock \doi{10.1007/BF02017927}.
\newblock \urlprefix\url{https://doi.org/10.1007/BF02017927}

\bibitem{DRITSCHEL}
Dritschel, M.A., Rovnyak, J.: Operators on indefinite inner product spaces.
\newblock In: Lectures on operator theory and its applications ({W}aterloo,
  {ON}, 1994), \emph{Fields Inst. Monogr.}, vol.~3, pp. 141--232. Amer. Math.
  Soc., Providence, RI (1996)

\bibitem{HANSEN}
Hansen, F.: Self-polar norms on an indefinite inner product space.
\newblock Publ. Res. Inst. Math. Sci. \textbf{16}(3), 889--913 (1980).
\newblock \doi{10.2977/prims/1195186935}.
\newblock \urlprefix\url{https://doi.org/10.2977/prims/1195186935}

\bibitem{Kaltenback}
Kaltenback, M., Woracek, H.: Theory of pontryagin spaces: Geometry and
  operators

\bibitem{LANGER}
Langer, M., Luger, A.: On norms in indefinite inner product spaces.
\newblock In: Recent advances in operator theory in {H}ilbert and {K}rein
  spaces, \emph{Oper. Theory Adv. Appl.}, vol. 198, pp. 259--264.
  Birkh\"{a}user Verlag, Basel (2010)

\bibitem{dirac}
M., D.P.A.: The physical interpretation of quantum mechanics.
\newblock Proc. Roy. Soc. London Ser. \textbf{180}, 1--40 (1942)

\end{thebibliography}

%
%



\end{document}